\documentclass[a4paper,11pt]{article}
\usepackage[utf8x]{inputenc}

\usepackage{a4wide}
\usepackage{amssymb}
\usepackage{amsmath}
\usepackage{amsthm}
\usepackage[mathscr]{euscript}
\usepackage[normalem]{ulem}
\usepackage[colorlinks=true]{hyperref}
\usepackage[usenames,dvipsnames]{xcolor}
\usepackage{appendix}
\usepackage{tikz}
\usetikzlibrary{arrows,angles}
\usetikzlibrary{matrix}
\usetikzlibrary{decorations}
\usetikzlibrary{arrows,calc,shapes,decorations.pathreplacing}
\usepackage{tikz-cd}
\usepackage{todonotes}

\numberwithin{equation}{section}

\theoremstyle{definition}

\newtheorem{Definition}{Definition}[section]

\newtheorem{Remark}[Definition]{Remark}

\theoremstyle{plain}

\newtheorem{Theorem}[Definition]{Theorem}

\newcommand{\R}{\mathbb R}

\usepackage[xcolor]{changebar}
\cbcolor{blue}
\newcommand{\Ric}{\mathrm{Ric}}

\usepackage[inline]{enumitem}
\newcommand{\enumlabelformat}{\roman}
\newcommand{\enumlabelfont}[1]{#1}
\newlength{\thelabelsep}
\setlist{labelsep=\thelabelsep}
\setlist[enumerate,1]{font=\enumlabelfont,label=(\enumlabelformat*),leftmargin=2.5em}
\setlist[itemize]{leftmargin=2.5em,label=$-$}

\newcounter{inlineenum}
\renewcommand{\theinlineenum}{\enumlabelformat{inlineenum}}

\let\epsilon\varepsilon
\let\phi\varphi

\def\sigdis#1{|#1|_{\pm}}

\newcommand{\ma}{\measuredangle}
\newcommand{\lm}[1]{\mathbb{L}^2(#1)}

\makeatletter
\let\save@mathaccent\mathaccent
\newcommand*\if@single[3]{%
  \setbox0\hbox{${\mathaccent"0362{#1}}^H$}%
  \setbox2\hbox{${\mathaccent"0362{\kern0pt#1}}^H$}%
  \ifdim\ht0=\ht2 #3\else #2\fi
  }
\newcommand*\rel@kern[1]{\kern#1\dimexpr\macc@kerna}
\newcommand*\widebar[1]{\@ifnextchar^{{\wide@bar{#1}{0}}}{\wide@bar{#1}{1}}}
\newcommand*\wide@bar[2]{\if@single{#1}{\wide@bar@{#1}{#2}{1}}{\wide@bar@{#1}{#2}{2}}}
\newcommand*\wide@bar@[3]{%
  \begingroup
  \def\mathaccent##1##2{%
    \let\mathaccent\save@mathaccent
    \if#32 \let\macc@nucleus\first@char \fi
    \setbox\z@\hbox{$\macc@style{\macc@nucleus}_{}$}%
    \setbox\tw@\hbox{$\macc@style{\macc@nucleus}{}_{}$}%
    \dimen@\wd\tw@
    \advance\dimen@-\wd\z@
    \divide\dimen@ 3
    \@tempdima\wd\tw@
    \advance\@tempdima-\scriptspace
    \divide\@tempdima 10
    \advance\dimen@-\@tempdima
    \ifdim\dimen@>\z@ \dimen@0pt\fi
    \rel@kern{0.6}\kern-\dimen@
    \if#31
      \overline{\rel@kern{-0.6}\kern\dimen@\macc@nucleus\rel@kern{0.4}\kern\dimen@}%
      \advance\dimen@0.4\dimexpr\macc@kerna
      \let\final@kern#2%
      \ifdim\dimen@<\z@ \let\final@kern1\fi
      \if\final@kern1 \kern-\dimen@\fi
    \else
      \overline{\rel@kern{-0.6}\kern\dimen@#1}%
    \fi
  }%
  \macc@depth\@ne
  \let\math@bgroup\@empty \let\math@egroup\macc@set@skewchar
  \mathsurround\z@ \frozen@everymath{\mathgroup\macc@group\relax}%
  \macc@set@skewchar\relax
  \let\mathaccentV\macc@nested@a
  \if#31
    \macc@nested@a\relax111{#1}%
  \else
    \def\gobble@till@marker##1\endmarker{}%
    \futurelet\first@char\gobble@till@marker#1\endmarker
    \ifcat\noexpand\first@char A\else
      \def\first@char{}%
    \fi
    \macc@nested@a\relax111{\first@char}%
  \fi
  \endgroup
}
\makeatother

\newcommand*{\md}{\mathrm{md}}

\title{The equivalence of smooth and synthetic notions of timelike sectional curvature bounds}

\author{Tobias Beran\thanks{{\tt tobias.beran@univie.ac.at}, Faculty of Mathematics, University of Vienna, Austria.}\\Michael Kunzinger\thanks{{\tt michael.kunzinger@univie.ac.at}, Faculty of Mathematics, University of Vienna, Austria.}\\Argam Ohanyan\thanks{{\tt argam.ohanyan@univie.ac.at}, Faculty of Mathematics, University of Vienna, Austria.}\\Felix Rott\thanks{{\tt frott@sissa.it}, SISSA, Trieste, Italy.}
}
\begin{document}

\date{\today}


\maketitle

\begin{abstract}

Timelike sectional curvature bounds play an important role in spacetime geometry, both for the understanding of classical smooth spacetimes and for the study of Lorentzian (pre-)length spaces introduced in \cite{kunzinger2018lorentzian}. In the smooth setting, a bound on the sectional curvature of timelike planes can be formulated via the Riemann curvature tensor. In the synthetic setting, bounds are formulated by comparing various geometric configurations to the corresponding ones in constant curvature spaces. The first link between these notions in the Lorentzian context was established in \cite{harris1982triangle}, which was instrumental in the proof of powerful results in spacetime geometry \cite{beem1985toponogov, beem1985decomposition, galloway2018existence}. For general semi-Riemannian manifolds, the equivalence between sectional curvature bounds and synthetic bounds was established in \cite{alexander2008triangle}, however in this approach the sectional curvatures of both timelike and spacelike planes have to be considered. In this article, we fill a gap in the literature by proving the full equivalence between sectional curvature bounds on timelike planes and synthetic timelike bounds on strongly causal spacetimes. As an essential tool, we establish Hessian comparison for the time separation and signed distance functions.
\vspace{1em}

\noindent
\emph{Keywords:} Timelike sectional curvature bounds, triangle comparison, hinge comparison
\medskip

\noindent
\emph{MSC2020:} 53B30, 53C23, 53B50

\end{abstract}

\section{Introduction}\label{section: intro}

Sectional curvature bounds are of great importance in both Riemannian and Lorentzian geometry. Spaces whose sectional curvature is bounded enjoy an abundance of geometric and topological properties which are usually quite rigid, i.e.\ slight variations on the curvature bound do not break these properties outside of some special cases.

In Riemannian geometry, Toponogov's comparison theorem and related results establish the equivalence between sectional curvature bounds and the comparison of various geometric configurations in the original space with those in the space of the corresponding constant sectional curvature. Some geometric configurations considered in this context are (geodesic) triangles, hinges, distances from a point to a geodesic, and so on. These descriptions of sectional curvature bounds, due to their independence of differential quantities, have led to a rich theory of metric spaces with sectional curvature bounds, i.e.\ Alexandrov spaces (see \cite{burago2001course} for an introductory text on the subject, and the references therein).

In the Lorentzian setting, timelike sectional curvature bounds have been considered as a global curvature condition stronger than the strong energy condition from General Relativity. The Lorentzian timelike sectional curvature comparison theory was initiated in \cite{harris1982triangle}, where it is shown that timelike sectional curvature bounds from  above imply a certain timelike/causal hinge comparison result. The methods established there have been crucial in the proofs of several important spacetime-geometric results, e.g.\ the proof of the Lorentzian splitting theorem for spacetimes with nonpositive timelike sectional curvature bounds containing a complete timelike line in \cite{beem1985toponogov, beem1985decomposition} and the proof of Bartnik's cosmological splitting conjecture \cite{bartnik1988remarks} under the stronger assumption of nonpositive timelike sectional curvature bounds in \cite{galloway2018existence}.

In the general setting of semi-Riemannian manifolds with arbitrary (constant) signature, an equivalence between bounds based on geometric triangle/angle/hinge comparison and a certain notion of sectional curvature comparison is available \cite{alexander2008triangle}. However, the methods employed in that work involve both timelike and spacelike planes, and considering bounds on spacelike planes in the Lorentzian context seems unnatural for the purposes of General Relativity and spacetime geometry.

Recently, an approach analogous to that of Alexandrov spaces has been initiated in the Lorentzian case (see \cite{kunzinger2018lorentzian}; see also \cite{minguzzi2022lorentzian} for a more recent local approach). The literature on Lorentzian synthetic spaces with timelike sectional curvature bounds is quite rich despite the field's rather recent initiation (see e.g.\ \cite{grant2019inextendibility, alexander2019generalized, beran02, rott2022gluing, beran2023splitting, beran2023alexandrov, beran2023toponogov}).

However, there is the following gap in the literature: The proof of \cite{harris1982triangle} only establishes that if $(M,g)$ is a globally hyperbolic spacetime of order $\sqrt{-K}$ whose timelike sectional curvature is bounded above by $K \in \R$, then it satisfies a global bound based on timelike hinge comparison by the same constant $K$. The converse direction as well as either direction for the opposite bounds do not seem to be available in the literature. The aim of this article is to fill this gap.

The paper is structured as follows: In Section \ref{section: variousnotionsofseccurvbounds}, we recall the notion of timelike sectional curvature bounds (Subsection \ref{Subsection: smoothtimelikeseccurvbounds}) on strongly causal spacetimes and establish a Hessian comparison result for the time separation and signed distance from a point. Then, in Subsection \ref{Subsection: synthetictimelikeseccurvbounds}, we discuss the various notions of synthetic timelike sectional curvature bounds, but restrict ourselves to strongly causal spacetimes (instead of general Lorentzian synthetic spaces) for the convenience of the reader. In Section \ref{section: equivalence_timelike_sec_curv_bounds}, we establish the aforementioned equivalence between smooth and synthetic timelike sectional curvature bounds, the direction synthetic $\Rightarrow$ smooth in Subsection \ref{subsection: synthetictosmooth} using Jacobi field computations, and the direction smooth $\Rightarrow$ synthetic in Subsection \ref{subsection: smoothtosynthetic} using Hessian comparison. In Subsection \ref{subsection: timelikeboundsandnontimelikeplanes}, we show via an explicit example that timelike sectional curvature bounds are strictly weaker than the corresponding semi-Riemannian bounds from \cite{alexander2008triangle}. Finally, in Section \ref{section: conclusion_outlook}, we summarize our work, give an improved version of the spacetime Reshetnyak gluing result from \cite{beran02} as an application, and discuss some related open problems.

\subsection{Notation and conventions} \label{subsection: notat_and_conv}

The symbol $\subset$ denotes (not necessarily strict) inclusion. Spacetime metrics are assumed to be smooth and have signature $(-,+,\dots,+)$. The Riemann curvature tensor of a semi-Riemannian manifold $(M,g)$ is $R_{X,Y}Z \equiv R(X,Y)Z = \nabla_X \nabla_Y Z - \nabla_Y \nabla_X Z - \nabla_{[X,Y]} Z$, where $X,Y,Z \in \mathfrak{X}(M)$ (note that this differs from \cite{alexander2008triangle,o1983semi} by a sign, but is consistent with \cite{harris1982triangle, beem1985toponogov, beem1985decomposition, BeemEhrlich, galloway2018existence}). We interchangeably write $\langle \cdot, \cdot \rangle$ for the metric $g$. The number $D_K$ is defined to be $+ \infty$ if $K \geq 0$ and $\pi/\sqrt{-K}$ if $K < 0$. A vector field $J$ along a geodesic $\gamma$ is called a \textit{Jacobi field} if $J'' + R_{J,\gamma'}\gamma' = 0$. For linear maps $A,B$ on a semi-Euclidean vector space $(V,\langle .,.\rangle)$, we write $A \geq B$ if $\langle Av,v\rangle \geq \langle Bv,v\rangle$ for all $v \in V$. We use standard notation from spacetime geometry (see \cite{BeemEhrlich, o1983semi}), i.e.\ $\ll$ and $\leq$ for the timelike resp.\ causal relation, $I^{\pm}$ and $J^{\pm}$ for timelike resp.\ causal futures/pasts, as well as $I(p,q) = I^+(p) \cap I^-(q)$ and $J(p,q) = J^+(p) \cap J^-(q)$ for timelike resp.\ causal diamonds, and $\tau$ for the time separation function.

\section{Various notions of timelike sectional curvature bounds} \label{section: variousnotionsofseccurvbounds}

In this section, we introduce and discuss various notions of timelike sectional curvature bounds on smooth spacetimes. Throughout, $(M,g)$ denotes a smooth, strongly causal spacetime of dimension $\geq 2$. The reason why we assume strong causality is because we require in an essential way the compatibility of the global time separation function with the local time separation on convex sets induced by the exponential map. 

\subsection{Smooth timelike sectional curvature bounds}
\label{Subsection: smoothtimelikeseccurvbounds}

We start with the classical smooth notion, which involves bounds on sectional curvatures of timelike planes. Recall that if $\{v,w\}$ span a non-degenerate $2$-plane in a tangent space $T_pM$, then the sectional curvature of $\Pi:= \mathrm{span}(v,w)$ is
\begin{align}
    K(\Pi) := \frac{\langle R(w,v)v,w\rangle}{\langle v,v \rangle \langle w,w\rangle - \langle v,w \rangle^2}\label{eq:secCurvOfPlane}
\end{align}

While our convention for the Riemann curvature tensor differs from that in \cite{alexander2008triangle, o1983semi} by a sign, the definition of sectional curvature agrees.\footnote{In \cite{alexander2008triangle, o1983semi}, the Riemann tensor in the numerator of \eqref{eq:secCurvOfPlane} appears in the form $\langle R(v,w)v,w \rangle$, hence the two minuses cancel out to give the same definition of sectional curvature that we use.}
Note that $K(\Pi)$ does not depend on the choice of basis vectors $\{v,w\}$.

\begin{Definition}[Smooth timelike sectional curvature bounds]
\label{Definition: smoothseccurvbounds}
We say that $(M,g)$ has \textit{smooth timelike sectional curvature bounded above (resp.\ below)} by $K \in \R$ if for all $p \in M$ and all timelike $2$-planes $\Pi \subset T_pM$, we have
\begin{align}
    K(\Pi) \leq K \quad (\text{resp. }K(\Pi) \geq K).
\end{align}
\end{Definition}

\begin{Remark}[On smooth timelike sectional curvature bounds]
\label{Remark: onsmoothseccurvbounds}
\begin{enumerate}
    \item[]
    \item Usually, one works with orthonormal bases of timelike tangent planes, i.e.\ a given timelike $2$-plane $\Pi \subset T_pM$ is spanned by $v,w \in T_pM$ such that $v$ is unit timelike ($\langle v,v \rangle = -1$), $w$ is unit spacelike ($\langle w,w \rangle = +1$) and $\langle v,w \rangle = 0$. In this case, \eqref{eq:secCurvOfPlane} simplifies to
    \begin{align}
        K(\Pi) = - \langle R(w,v)v,w \rangle.
    \end{align}
    \item If $(M,g)$ has smooth timelike sectional curvature bounded above (resp.\ below) by $K \in \R$, then for all $v \in TM$ timelike
    \begin{align}
    \Ric(v,v) \geq (n-1) K \langle v,v \rangle  \quad (\text{resp. } \Ric(v,v) \leq (n-1)K \langle v,v \rangle).
    \end{align}
    Indeed: For definiteness, suppose the bound is from above by $K$ and let $v \in T_pM$ be any timelike tangent vector. W.l.o.g.\ we may assume that $\langle v,v \rangle = -1$. Complete $\{v\}$ to an orthonormal basis $\{v,e_1,\dots,e_{n-1}\}$ of $T_pM$, where $n:=\dim M$. Writing $e_0:=v$, we have $\langle R(e_i,v)v,e_i \rangle \geq -K$ ($i=1,\dots,n-1$) by the curvature bound, hence
    \begin{align*}
        \Ric(v,v) = \sum_{i=0}^{n-1} \langle e_i,e_i\rangle \langle R(e_i,v)v,e_i \rangle = \sum_{i=1}^{n-1} \langle R(e_i,v)v,e_i \rangle \geq -(n-1)K.
    \end{align*}
    \item In \cite{alexander2008triangle}, sectional curvature bounds on general semi-Riemannian manifolds are defined as follows: A semi-Riemannian manifold $(M,g)$ is said to have sectional curvature bounded below (above) by $K \in \R$ if for all $p \in M$ and all $v,w \in T_pM$,
    \begin{align}
        \langle R(w,v)v,w \rangle \geq K(\langle v,v \rangle \langle w,w \rangle - \langle v,w \rangle^2) \quad (\text{resp. } \leq).
    \end{align}
    Note that a sectional curvature bound below (above) by $K$ in this sense implies a smooth timelike sectional curvature bound above (below) by $K$. 
\end{enumerate}
\end{Remark}

Due to Remark \ref{Remark: onsmoothseccurvbounds}(ii), a smooth timelike sectional curvature bound from  \textbf{above} (resp.\ \textbf{below}) by $K \in \R$ corresponds geometrically (viewed from the perspective of positive definite, i.e.\ Riemannian, geometry) to a  ``curvature bound" from \textbf{below} (resp.\ \textbf{above}) by $-K \in \R$. As we will see, timelike sectional curvature bounds \textbf{above} (resp.\ \textbf{below}) by $K$ correspond to synthetic timelike bounds \textbf{below} (resp.\ \textbf{above}) by $K$. This mismatch is due to the fact that the conventions in the synthetic literature were chosen to align with \cite{alexander2008triangle}, cf.\ Remark \ref{Remark: onsmoothseccurvbounds}(iii).

Let us now turn to an important consequence of smooth timelike sectional curvature bounds, namely Hessian comparison for the time separation and signed distance. Related results are readily available in the literature (cf.\ \cite{alexander2008triangle, treude2013volume}), but we do not know of any source which contains this exact analogue of the well-known Riemannian result.

We will need the notion of \emph{signed distance} from \cite{alexander2008triangle}: Given a convex set $U \subset M$ and $p,q \in U$, the signed distance between $p$ and $q$ is $\sigdis{pq}:= \pm|\exp_p^{-1}(q)|_g$, where the sign is ``$-$'' if $p \ll q$ or $q \ll p$, and ``$+$'' otherwise.

\begin{Definition}[Modified distance functions]
Given $K \in \R$, the \textit{modified distance function} $\md^K_\tau:[0,D_K) \to \R$ is defined as
\begin{align}\label{eq:signed_distance_tau}
    \md^K_\tau(t):= \begin{cases}
        \frac{t^2}{2} & K = 0,\\
        \frac{\cosh(\sqrt{K}t) - 1}{K} & K > 0,\\
        \frac{\cos(\sqrt{-K}t) - 1}{K} & K < 0.
    \end{cases}
\end{align}
This definition is tailored to the time separation function. For some of our results it will be necessary to switch to the signed distance. 
This is why we introduce the \emph{extended modified distance function}
$\md^K_S:(-D_K,D_{-K}) \to \R$ 
\begin{align}
    \md^K_S(t):= \begin{cases}
        \md^{-K}_\tau(t) & t\ge 0,\\
        -\md^K_\tau(-t) & t < 0.
    \end{cases}
\end{align}
\end{Definition}
This definition reflects the anti-isometry between de Sitter space and anti-de Sitter space (resp.\ from Minkowski space to itself): it sends spacelike vectors in de Sitter space to timelike vectors of the same length in anti-de Sitter space and vice versa.

It is easy to check that on $(-D_K,0]$, $\md^K_S$ is the unique solution of the initial value problem
\begin{align*}
    \begin{cases}
        (\md^K_S)'' - K \md^K_S = -1,\\
        \md^K_S(0) = 0,\\
        (\md^K_S)'(0) = 0.
    \end{cases}
\end{align*}
and on $[0,D_{-K})$, 
$\md^K_S$ is the unique solution of the initial value problem
\begin{align*}
    \begin{cases}
        (\md^K_S)'' + K \md^K_S = 1,\\
        \md^K_S(0) = 0,\\
        (\md^K_S)'(0) = 0.
    \end{cases}
\end{align*}

\begin{Theorem}[Lorentzian Hessian comparison]
\label{Theorem: Hessiancomparison}
Let $(M,g)$ be a strongly causal spacetime with smooth timelike sectional curvature bounded above (resp.\ below) by $K \in \R$, $p \in M$ and $U$ a convex neighborhood of $p$.
Then we have the following inequalities for $(1,1)$-Hessians: 
\begin{enumerate}
\item On $I^+(p) \cap U$, for \emph{any} vector $v$, we have
\begin{equation}
    \label{eq: signeddistanceHessiancomp}   
    \left<\mathrm{Hess}\, \md^K_S(\sigdis{p\cdot})(v),v\right> \leq (1-K\md^K_S(|p\cdot|_\pm))\left<v,v\right> \quad (\text{resp. } \geq)\,, 
\end{equation}
or equivalently
\begin{equation}
\label{eq: Hessiancomp}
    \mathrm{Hess}\, \tau(p,\cdot) \geq - \frac{(\md^K_\tau)''(\tau(p,\cdot))}{(\md^K_\tau)'(\tau(p,\cdot))} \pi_{\tau(p,.)} \quad (\text{resp. } \leq)\,,
\end{equation}
where $\pi_{\tau(p,.)}$ denotes the projection onto the tangent spaces of level sets of $\tau(p,.)$, more concretely $\pi_{\tau(p,.)} = \mathrm{Id} + \langle \nabla \tau(p,.),.\rangle \nabla \tau(p,.)$. 

If, in addition, $(M,g)$ is globally hyperbolic if $K \geq 0$ or order ($\sqrt{-K}$)-globally hyperbolic\footnote{Global hyperbolicity of order $\sqrt{-K}$ means that for any $x\ll y$ with $\tau(x,y) < D_K$ the causal diamond $J(x,y)$ is compact. This notion was introduced in \cite{harris1982triangle}.} if $K < 0$, then these inequalities hold on $I^+(p) \setminus \mathrm{Cut}^+(p)$ if $K\geq 0$, and on $\{q \in I^+(p) : \tau(p,q) < D_K\} \setminus \mathrm{Cut}^+(p)$ if $K < 0$.
\item On all of $U$, for all \emph{timelike} vectors $v$, we have 
\begin{equation}
\label{eq: signeddistanceHessiancomp2}
\left<\mathrm{Hess}\, \md^K_S(\sigdis{p\cdot})(v),v\right> \leq (1-K \md^K_S(|p\cdot|_\pm))\left<v,v\right> \quad (\text{resp. } \geq)\,.
\end{equation}
\end{enumerate}
The time reversed version of (i) also holds.
\end{Theorem}
\begin{proof}
The proof proceeds by adapting arguments from \cite{alexander2008triangle}, to which we refer for notations used below.

For (i), note that the proof of \cite[Prop.\ 5.2]{alexander2008triangle} only uses curvature comparison in applying \cite[Cor.\ 4.6]{alexander2008triangle}, and the proof of that result only uses $R_{\sigma'}\geq \tilde{R}_{\tilde{\sigma}'}$ where $\sigma$ is timelike in the case we are considering.

For the part of (ii) that is not covered by (i), we only have $\langle R_{\sigma'}(v),v\rangle \geq \langle \tilde{R}_{\tilde{\sigma}'}(v),v\rangle$ for timelike vectors $v$. We replace \cite[Cor.\ 4.5]{alexander2008triangle} by claiming that $\langle R_{\sigma'}(v),v\rangle \geq \langle \tilde{R}_{\tilde{\sigma}'}(v),v\rangle$ for all timelike vectors $v$ implies $\langle S_{K,q}(v),v\rangle \leq \langle \tilde{S}_{K,q}(v),v\rangle$ for all timelike vectors $v$. Note that the proof of \cite[Cor.\ 4.5]{alexander2008triangle} still works up to \cite[Thm.\ 4.3]{alexander2008triangle}, which we again have to replace by the claim that $\langle R_1(v),v\rangle \leq \langle R_2(v),v\rangle$ for all timelike vectors $v$ implies $\langle S_1(v),v\rangle \geq \langle S_2(v),v\rangle$ for all timelike vectors $v$. To obtain this, note that in the argument by contradiction in the proof of \cite[Thm.\ 4.3]{alexander2008triangle} the vector $x_0$ with $\langle(S_1(t_0)-S_\delta(t_0))x_0,x_0\rangle=0$ is timelike in our case. Then indeed the rest of the proof of \cite[Thm.\ 4.3]{alexander2008triangle} goes through unchanged and gives our desired inequality between $S_1$ and $S_2$.
\end{proof}

\begin{Remark}[D'Alembertian comparison]
Suppose that $(M,g)$ has smooth timelike sectional curvature bounded above (resp.\ below) by $K \in \R$. By taking the trace of the inequality \eqref{eq: Hessiancomp}, one obtains the well-known d'Alembertian comparison inequality
\begin{align}
    \square \tau(p,.) \geq - (n-1) \frac{(\md^K_\tau)''(\tau(p,.))}{(\md^K_\tau)'(\tau(p,.))} \quad (\text{resp.\ } \leq),
\end{align}
which is known to hold under the weaker assumption of $\Ric(v,v) \geq (n-1)K \langle v,v\rangle$ (resp.\ $\leq$) for all $v \in TM$ timelike (see e.g.\ \cite[Thm.\ 3.3.5]{treude2011ricci}).
\end{Remark}

\subsection{Synthetic timelike sectional curvature bounds}
\label{Subsection: synthetictimelikeseccurvbounds}

Let us now move on to a discussion of various synthetic notions of timelike sectional curvature bounds. The study of such notions goes back to \cite{kunzinger2018lorentzian}, where Lorentzian (pre-)length spaces were introduced as synthetic analogues of spacetimes. There, synthetic timelike curvature bounds were studied via timelike and causal triangles. Since then, notions based on angles, hinges, four-point conditions and concavity/convexity of the time separation have been introduced  \cite{beran2022angles, barrera2022comparison, beran2023curvature}. In \cite{beran2023curvature}, the authors  study the interrelationship between these various synthetic notions of timelike curvature bounds, and show the equivalence of all of these notions under mild assumptions. Since we restrict ourselves to the study of classical spacetimes in this article, we will state all of the definitions and results obtained in the aforementioned papers only for spacetimes. As before, $(M,g)$ is a strongly causal spacetime of dimension $n:=\dim M \geq 2$. Let us first introduce some preliminary notions.

\begin{Definition}[Constant curvature spaces]
We denote by $\mathbb{L}^2(K)$ the simply connected $2$-dimensional Lorentzian space form of constant sectional curvature $K$, see \cite[Def.\ 4.5]{kunzinger2018lorentzian}. The number $D_K$ is the timelike diameter of its maximal globally hyperbolic subsets. We will denote by $\bar{\tau}$ the time separation function on $\mathbb{L}^2(K)$.
\end{Definition}

Next, we introduce the synthetic notions needed to define synthetic timelike sectional curvature bounds.

\begin{Definition}[Triangles, angles, hinges]
\label{Definition: trianglesandangles}
Let $U \subset M$ be a convex set.
\begin{enumerate}
\item A \textit{timelike triangle} in $U$ is a triple $(p,q,r) \in U^3$ such that $p \ll q \ll r$. A \textit{causal triangle} in $U$ is a triple $(p,q,r) \in U^3$ such that $p \ll q \leq r$ or $p \leq q \ll r$. We will also speak of timelike or causal triangles when we mean the connected radial geodesics between $p,q,r$, and we will write $\Delta(p,q,r) \subset U$, in which case we refer to the connecting geodesics as the \textit{sides} of the triangle. \textit{Triangle} will always refer to timelike or causal triangles.
\item A \textit{comparison triangle} for a triangle $\Delta (p,q,r)$ in $U$ is a triple $(\bar{p},\bar{q},\bar{r})$ in $\mathbb{L}^2(K)$ with $\tau(p,q) = \bar{\tau}(\bar{p},\bar{q})$, $\tau(q,r) = \bar{\tau}(\bar{q},\bar{r})$ and $\tau(p,r) = \bar{\tau}(\bar{p},\bar{r})$. When $K < 0$, we suppose that $U$ is chosen small enough such that $\sup_{p,q \in U} \tau(p,q) < D_K$. This ensures that $\Delta(\bar{p},\bar{q},\bar{r})$ is a triangle in a convex subset of $\mathbb{L}^2(K)$. With these conventions, for any triangle $\Delta(p,q,r) \subset U$ there exists a unique (up to isometry) comparison triangle in $\mathbb{L}^2(K)$
(cf.\ \cite[Lem.\ 2.1]{alexander2008triangle}). Given $x \in \Delta(p,q,r)$ on a timelike side, we refer to the unique point $\bar{x}$ on the corresponding timelike side of $\Delta(\bar{p},\bar{q},\bar{r})$ with the same time separation to the endpoints of the side as the \textit{corresponding point}.
\item Given a triangle $\Delta(p,q,r) \subset U$ such that $p \ll q$ and its comparison triangle $\Delta(\bar{p},\bar{q},\bar{r}) \subset \mathbb{L}^2(K)$, the \textit{$K$-comparison angle at $p$} is
\begin{align*}
    \tilde{\measuredangle}_p^K(q,r):= \measuredangle_{\bar{p}}^{\mathbb{L}^2(K)}(\bar{q},\bar{r}) = \mathrm{arcosh}(|\langle \gamma_{\bar{p}\bar{q}}'(0),\gamma_{\bar{p}\bar{r}}'(0)\rangle|),
\end{align*}
where $\gamma_{\bar{p}\bar{q}}$ and $\gamma_{\bar{p}\bar{r}}$ are the unique $\bar{\tau}$-arclength parametrized timelike maximizing geodesics in $\mathbb{L}^2(K)$ from $\bar{p}$ to $\bar{q}$ and $\bar{p}$ to $\bar{r}$, respectively. The \textit{sign} $\sigma$ of a $K$-comparison angle is the sign of the inner product between the tangent vectors of the geodesics above. The \textit{signed $K$-comparison angle} $\tilde{\measuredangle}_p^{K,S}(q,r)$ is defined as the product between the $K$-comparison angle and its sign.
\item Given $p \in U$ and $\alpha,\beta$ unit speed future or past directed timelike geodesics from $p$ contained in $U$, the \textit{angle at $p$ between $\alpha$ and $\beta$} is $\measuredangle_p(\alpha,\beta):= \mathrm{arcosh}(|\langle \alpha'(0),\beta'(0)\rangle|)$. It is precisely the limit of $K$-comparison angles $\tilde{\measuredangle}_p^K(\alpha(s),\beta(t))$ as $(s,t) \to (0,0)$ with $(s,t)$ in a set such that $\Delta(p,\alpha(s),\beta(t))$ is a timelike (or causal) triangle, independently of $K \in \R$ \cite{beran2022angles}. The \textit{sign} $\sigma$ of an angle and the \textit{signed angle} $\measuredangle^S_p(\alpha,\beta)$ are defined in the analogous way.
\item A \textit{hinge} in $U$ is a point $p \in U$ together with two radial timelike geodesics $\alpha,\beta$ from $p$. A \textit{comparison hinge} in $\mathbb{L}^2(K)$ is a point $\bar{p}$ together with two radial timelike geodesics $\bar{\alpha},\bar{\beta}$ from $p$ of the same length and same time orientation as $\alpha,\beta$, such that $\measuredangle_p(\alpha,\beta) = \measuredangle_{\bar{p}}(\bar{\alpha},\bar{\beta})$.
\end{enumerate}
\end{Definition}

\begin{Definition}[Synthetic timelike sectional curvature bounds]
\label{Definition: synthetictimelikeseccurvbounds}
We say that $(M,g)$ has \textit{synthetic timelike sectional curvature bounded below (resp.\ above)} by $K \in \R$ in any of the following senses if $M$ is covered by convex sets $U$ (if $K < 0$, suppose in addition $\sup_{q_1,q_2 \in U} \tau(q_1,q_2) < D_K$) such that the corresponding properties listed below are satisfied:
\begin{enumerate}
    \item \textit{Timelike triangle comparison}: Given any timelike triangle $\Delta(p,q,r) \subset U$ and its comparison triangle $\Delta(\bar{p},\bar{q},\bar{r}) \subset \mathbb{L}^2(K)$, for any $x,y$ on the sides and corresponding points $\bar{x},\bar{y}$, we have
    \begin{align}
        \tau(x,y) \leq \bar{\tau}(\bar{x},\bar{y}) \quad (\text{resp.\ } \geq).
    \end{align}
    \item \textit{Monotonicity comparison}: Given two radial timelike geodesics $\alpha:[0,a] \to U$ and $\beta:[0,b] \to U$ starting at $x:=\alpha(0)=\beta(0)$, the function $\theta(s,t):=\tilde{\measuredangle}_x^{K,S}(\alpha(s),\beta(t))$, defined for all $(s,t) \in (0,a] \times (0,b]$ such that $\alpha(s) \leq \beta(t)$, is monotonically increasing in both $s$ and $t$ (resp.\ decreasing).
    \item \textit{Angle comparison}: Given two radial timelike geodesics $\alpha:[0,a] \to U$ and $\beta:[0,b] \to U$ starting at $x:=\alpha(0) = \beta(0)$ such that $\alpha(a) \leq \beta(b)$, we have
    \begin{align}
        \measuredangle^S_x(\alpha,\beta) \leq \tilde{\measuredangle}_x^{K,S}(\alpha(a),\beta(b)) \quad (\text{resp.\ } \geq).
    \end{align}
    \item \textit{Hinge comparison}: Given two radial timelike geodesics $\alpha:[0,a] \to U$ and $\beta:[0,b] \to U$ starting at $x:=\alpha(0) = \beta(0)$, consider a comparison hinge $(\bar{\alpha},\bar{\beta})$ in $\mathbb{L}^2(K)$. Then
    \begin{align}
        \tau(\alpha(a),\beta(b)) \geq \bar{\tau}(\bar{\alpha}(a),\bar{\beta}(b)) \quad (\text{resp.\ } \leq).
    \end{align}
    \item \textit{Causal triangle comparison}: Given any causal triangle $\Delta(p,q,r) \subset U$ and its comparison triangle $\Delta(\bar{p},\bar{q},\bar{r}) \subset \mathbb{L}^2(K)$, for any $x,y$ on the sides of $\Delta(x,y,z)$ and corresponding points $\bar{x},\bar{y} \in \Delta(\bar{p},\bar{q},\bar{r})$ we have
    \begin{align}
        \tau(p,q) \leq \bar{\tau}(\bar{x},\bar{y}) \quad (\text{resp.\ } \geq).
    \end{align}
    \item \textit{Convexity/concavity of the time separation}: Given any $p \in U$ and any timelike geodesic segment $\gamma:[a,d] \to U$, the function $f:D \to \R$, where $D:=\{t \in [a,d] : \gamma(t) \in J^+(p) \cup J^-(p)\}$, defined by
    \begin{align}
        f(t):=\begin{cases}
            \md^K_\tau(\tau(p,\gamma(t))), & p \leq \gamma(t),\\
            \md^K_\tau(\tau(\gamma(t),p)), & \gamma(t) \leq p,
        \end{cases}
    \end{align}
    can be extended to all of $[a,d]$ in such a way that it satisfies the following differential inequality in the distributional sense:
    \begin{align}
        f'' - Kf \geq 1 \quad (\text{resp.\ } \leq).
    \end{align}
\end{enumerate}
\end{Definition}

\begin{Theorem}[Equivalence of synthetic bounds]
\label{Theorem: Equivalenceofsyntheticbounds}
Let $(M,g)$ be a strongly causal spacetime. Then all of the notions of synthetic timelike sectional curvature bounds below (resp.\ above) by $K \in \R$ introduced in Definition \ref{Definition: synthetictimelikeseccurvbounds} are equivalent.
\begin{proof}
    This follows from \cite[Thm.\ 5.1]{beran2023curvature} by noting that the Lorentzian pre-length space structure induced by a strongly causal spacetime is regular, locally causally closed and locally causally $D_K$-geodesic in the sense of the definitions appearing in that reference.
\end{proof}
\end{Theorem}

\begin{Remark}[On synthetic bounds]
\label{Remark: onsyntheticbounds}
For simplicity of presentation, we did not include all notions of synthetic timelike sectional curvature bounds discussed in \cite{beran2023curvature}. There the authors consider (in addition to what we have presented in Definition \ref{Definition: synthetictimelikeseccurvbounds}) one-sided versions of the various conditions, e.g.\ in the triangle comparisons the time separation from an arbitrary point to a vertex, four-point configurations, as well as strict versions of curvature bounds, which amount to comparing the \textit{extended time separation function} $l$ which is defined as $l(x,y) :=\tau(x,y)$ if $x \leq y$ and $l(x,y) := - \infty$ otherwise. All of these versions are again equivalent among each other and to the ones we have discussed on strongly causal spacetimes. We refer to \cite{beran2023curvature} for details.
\end{Remark}

Let us close our discussion on synthetic timelike sectional curvature bounds by addressing the question of globalization of synthetic timelike curvature bounds
by $K\in \R$. To simplify the discussion, for the remainder of this section we restrict our attention to globally hyperbolic spacetimes if $K\ge 0$ and order ($\sqrt{-K}$)-globally hyperbolic spacetimes if $K<0$. We are interested in the following question: 
When are the comparison conditions in Definition \ref{Definition: synthetictimelikeseccurvbounds} satisfied for arbitrarily large synthetic objects instead of local ones in a convex neighborhood? This question has been answered for upper curvature bounds in \cite{beran2023alexandrov} and for lower curvature bounds in \cite{beran2023toponogov} in the more general context of Lorentzian (pre-)length spaces. We summarize the main results in the following Theorem. 
By \textit{global synthetic timelike sectional curvature bounds above resp.\ below by $K\in \R$} we will mean that 
 any of the conditions in Definition \ref{Definition: synthetictimelikeseccurvbounds} 
hold for arbitrarily large objects, e.g.\ global timelike triangle comparison means for any global timelike triangle (i.e.\ choice of points $p \ll q \ll r$, together with a choice of maximizing timelike geodesics between them, such that there exists a comparison triangle in $\mathbb{L}^2(K)$) the comparison inequality holds, and so on. All of the global versions of the notions described in Definition \ref{Definition: synthetictimelikeseccurvbounds} are again equivalent due to \cite[Thm.\ 5.1]{beran2023curvature} (although that result only refers to local bounds, the arguments used in its proof apply to global comparison situations just as well, given our global hyperbolicity assumptions on $M$).

\begin{Theorem}[Globalization of synthetic timelike sectional curvature bounds, \cite{beran2023alexandrov, beran2023toponogov}]
\label{Theorem: Globalization}
Let $K\in \R$ and let $(M,g)$ be a globally hyperbolic spacetime if $K\ge 0$ and an order ($\sqrt{-K}$)-globally hyperbolic spacetime if $K<0$.
\begin{enumerate}
    \item Suppose $(M,g)$ has a synthetic timelike sectional curvature bound from above by $K \in \R$ in any of the equivalent senses of Definition \ref{Definition: synthetictimelikeseccurvbounds}. In addition, suppose the following conditions hold:
    \begin{itemize}
        \item[$\bullet$] For any $(x,y) \in \tau^{-1}((0,D_K))$ there exists a \emph{unique} timelike maximizing geodesic $\gamma_{xy}$ connecting them (parametrized proportional to $\tau$-arclength on $[0,1]$).
        \item[$\bullet$] The geodesic map $\tau^{-1}((0,D_K)) \times [0,1] \to M$, $(x,y,t) \mapsto \gamma_{xy}(t)$ is continuous.
    \end{itemize}
    Then $(M,g)$ has global synthetic timelike sectional curvature bounded above by $K$.
\item Suppose $(M,g)$ has synthetic timelike sectional curvature bounded below by $K \in \R$ in any of the equivalent senses of Definition \ref{Definition: synthetictimelikeseccurvbounds}. Then $(M,g)$ has global synthetic timelike sectional curvature bounded below by $K$.
\end{enumerate}
\end{Theorem}

\section{Equivalence of timelike sectional curvature bounds}
\label{section: equivalence_timelike_sec_curv_bounds}

In this section, we give the proof of the equivalence between smooth and synthetic timelike sectional curvature bounds in full detail. We also show by an explicit example that smooth timelike sectional curvature bounds are strictly weaker than the semi-Riemannian bounds described in \cite{alexander2008triangle}.

\subsection{From synthetic to smooth bounds}
\label{subsection: synthetictosmooth}

\begin{Theorem}
\label{Theorem: syntheticimpliessmooth}
Let $(M,g)$ be a strongly causal spacetime with synthetic timelike sectional curvature bounded below (resp.\ above) by $K\in\R$ in any of the equivalent senses of Definition \ref{Definition: synthetictimelikeseccurvbounds}. Then $(M,g)$ has smooth timelike sectional curvature bounded above (resp.\ below) by $K \in \R$ in the sense of Definition \ref{Definition: smoothseccurvbounds}.
\end{Theorem}
\begin{proof}
The proof follows ideas in \cite{alexander2008triangle}, but we have to modify the setup to use timelike hinge comparison. 
Let $p\in M$ and let $\Pi \subset T_pM$ be a timelike plane. Let $\{v,w\}$ be an orthonormal basis of $\Pi$, with $v$ timelike. Consider the future directed timelike unit speed geodesic $\gamma$ defined by the initial conditions $\gamma(0)=p$, $\gamma'(0)=v$. We extend it to a geodesic variation as follows: It is easy to check that $2v - w$ is a future directed timelike vector, and we define
\begin{equation*}
h(s,t) :=\exp_p(t(v+s(2v - w)))
\end{equation*}
Clearly, $h$ is a geodesic variation of $\gamma$ by future directed timelike geodesics, thus the variation field $J(t):=\partial_s|_{s=0}h(s,t)$ is a Jacobi field along $\gamma$ satisfying $J(0)=0$, $J'(0)=2v - w$. 

Similarly, in $\lm{K}$, consider a future directed timelike unit speed geodesic $\tilde{\gamma}$ with $\tilde{p}:=\tilde{\gamma}(0)$, $\tilde{v}:=\tilde{\gamma}'(0)$, and let $\tilde{w}\in T_{\tilde{p}}\lm{K}$ be a unit spacelike vector orthogonal to $\Tilde{v}$. We define the analogous geodesic variation in $\mathbb{L}^2(K)$ via
\begin{equation*}
\tilde{h}(s,t)=\exp_{\tilde{p}}(t(\tilde{v}+s(2 \tilde{v} - \tilde{w})))
\end{equation*}
As before, $\tilde{h}$ is a geodesic variation of $\tilde{\gamma}$ by future directed timelike geodesics, hence its variation field $\tilde{J}(t):=\partial_s|_{s=0}\tilde{h}(s,t)$ is a Jacobi field along $\tilde{\gamma}$ satisfying $\tilde{J}(0)=0$, $\tilde{J}'(0)=2\Tilde{v} - \tilde{w}$. Note that for each fixed $s$, $h(s,\cdot)$ and $\tilde{h}(s,\cdot)$ have the same (not necessarily unit) speed. Due to this, as well as the orthonormality of $\{v,w\}$ and $\{\tilde{v},\tilde{w}\}$, it is easy to see that $\ma_{\tilde{p}}(\tilde{\gamma},\tilde{h}(s,\cdot))=\ma_{p}(\gamma,h(s,\cdot))$ and thus $(\tilde{\gamma}|_{[0,a]},\tilde{h}(s,\cdot)|_{[0,b]})$ forms a comparison hinge to $(\gamma|_{[0,a]},h(s,\cdot)|_{[0,b]})$ (for any $a,b > 0$ such that the geodesics in question exist up to these parameters). Additionally, one may verify explicitly in $\mathbb{L}^2(K)$ that for small enough $s,t>0$ we have that $\tilde{h}(s,t)\ll\tilde{\gamma}(t)$.
We may thus apply hinge comparison to obtain
\begin{equation}\label{eq:hinge}
\tau(h(s,t),\gamma(t))\geq \bar{\tau}(\tilde{h}(s,t),\tilde{\gamma}(t))\quad (\text{resp.\ }\leq).
\end{equation}
By positivity of the right hand side also $h(s,t) \ll \gamma(t)$ if the curvature is bounded from below. 
In the case of curvature bounded above we cannot conclude $h(s,t) \ll \gamma(t)$, but the reverse triangle inequality guarantees $\gamma(t) \not \ll h(s,t)$: Indeed, we have 
\begin{equation*}
\underbrace{\bar{\tau}(\tilde{p},\tilde{\gamma}(t))}_{=\tau(p,\gamma(t))}\geq \underbrace{\bar{\tau}(\tilde{p},\tilde{h}(s,t))}_{=\tau(p,h(s,t))}+\underbrace{\bar{\tau}(\tilde{h}(s,t),\tilde{\gamma}(t))}_{>0}\,,
\end{equation*}
contradicting the reverse triangle inequality $\tau(p,h(s,t))\geq\tau(p,\gamma(t))+\tau(\gamma(t),h(s,t))$ which would need to hold if $\gamma(t)\ll h(s,t)$ were true.

As we know $\tilde{h}(s,t)\ll\tilde{\gamma}(t)$, and either $h(s,t) \ll \gamma(t)$ if the curvature bound is from below by $K$ or $\gamma(t)\not\ll h(s,t)$ if it is from above by $K$, we may formulate the above considerations via the signed distance as
\begin{equation*}
\sigdis{h(s,t)\gamma(t)}\leq \sigdis{\tilde{h}(s,t)\tilde{\gamma}(t)}\quad (\text{resp. }\geq).
\end{equation*}

This leads to the analogous inequality on the signed norm of the Jacobi fields:
\begin{equation*}
\sigdis{J(t)}=\lim_{s\searrow 0}\frac{\sigdis{\gamma(t)h(s,t)}}{s}\leq \lim_{s\searrow 0}\frac{\sigdis{\tilde{\gamma}(t)\tilde{h}(s,t)}}{s} = \sigdis{\tilde{J}(t)}\quad (\text{resp. }\geq)
\end{equation*}
and the same inequality for $\langle J(t),J(t)\rangle$ and $\langle \tilde{J}(t),\tilde{J}(t)\rangle $, i.e.\ $\langle J(t), J(t) \rangle \leq \langle \tilde{J}(t) , \tilde{J}(t) \rangle$.

Note that $J$ satisfies the following set of equalities:
\begin{align*}
J'(0)&=2v - w,\\
J''&=-R_{J, \gamma'}\gamma',\\
J''(0)&=0,\\
J'''&=-R_{J, \gamma'}\gamma'-R_{J', \gamma'}\gamma',\\
J'''(0)&=-R_{2v - w, v}v.
\end{align*}
The analogous equalities hold for $\tilde{J}$ in $\mathbb{L}^2(K)$.

If we develop the inequality $\langle J(t),J(t)\rangle \leq \langle \tilde{J}(t),\tilde{J}(t) \rangle $ in orders of $t$, all of the terms up to order $t^3$ agree, and the term of order $t^4$ yields the following inequality, where the Riemann curvature tensor $\tilde{R}$ on $\mathbb{L}^2(K)$ is of the form $\Tilde{R}(u_1,u_2)u_3 = K (\langle u_2,u_3 \rangle u_1 - \langle u_1, u_3 \rangle u_2)$ due to the constant sectional curvature:
\begin{align*}
\langle -R_{2v - w, v}v,2v - w \rangle &\leq \langle -\tilde{R}_{2 \tilde{v} - \tilde{w},\tilde{v}}\tilde{v},2 \tilde{v} - \tilde{w} \rangle \quad (\text{resp. }\geq)\\
&=-K (\langle \tilde{v},\tilde{v}\rangle \langle 2 \tilde{v} - \tilde{w}, 2 \tilde{v} - \tilde{w}\rangle - \langle 2 \tilde{v} - \tilde{w},\tilde{v} \rangle^2).
\end{align*}
Note that $\langle u_1,u_2\rangle^2 > \langle u_1, u_1 \rangle \langle u_2,u_2 \rangle$ for any pair of linearly independent timelike vectors by the reverse Cauchy-Schwarz inequality. Due to this and the fact that we may replace the products on the right hand side by the corresponding ones of $v,w$ in $T_pM$, 
we conclude
\begin{align*}
    \frac{\langle R_{2v - w, v}v,2v - w \rangle}{\langle v,v\rangle \langle 2v - w, 2 v - w\rangle - \langle 2 v - w,v \rangle^2} \leq K \quad (\text{resp. } \geq).
\end{align*}
Since $\{v,2v - w\}$ is a basis of $\Pi$, we conclude that the sectional curvature of $\Pi$ is bounded above (resp.\ below) by $K$, which gives the claim.
\end{proof}

\subsection{From smooth to synthetic bounds}
\label{subsection: smoothtosynthetic}

\begin{Theorem}
\label{Theorem: Smoothimpliessynthetic}
Let $(M,g)$ be a strongly causal spacetime with smooth timelike sectional curvature bounded above (below) by $K\in\R$ in the sense of Definition \ref{Definition: smoothseccurvbounds}. Then $(M,g)$ has synthetic timelike sectional curvature bounded below (above) by $K$ in any of the equivalent senses of Definition \ref{Definition: synthetictimelikeseccurvbounds}.
\end{Theorem}
\begin{proof}
We verify (vi) in Definition \ref{Definition: synthetictimelikeseccurvbounds}. The proof is a consequence of Hessian comparison for the signed distance, cf.\ Theorem \ref{Theorem: Hessiancomparison}(ii). For definiteness, suppose that the smooth timelike sectional curvature is bounded above by $K$. Fix $p \in M$ and let $U$ be a convex neighborhood of $p$. Let $\gamma:[0,b] \to U$ be the unique $\tau$-arclength parametrized timelike maximizer from $q_-:=\gamma(0)$ to $q_+:=\gamma(b)$. 

Consider the function $f(t):=\md^K_S(\sigdis{p\gamma(t)})$, which is smooth on $[0,b]$: Indeed it can be written as a convergent power series in $\langle \exp_p^{-1}(\gamma(t)), \exp_p^{-1}(\gamma(t)) \rangle$, cf.\ \cite[Eq.\ (1.6)]{alexander2008triangle}.
Since $-f$ is an extension of the function in Definition \ref{Definition: synthetictimelikeseccurvbounds} (vi), we are done if we can show that $f'' - Kf \leq -1$. We conclude our desired result from Theorem \ref{Theorem: Hessiancomparison}(ii):
\begin{align*}
f''(t)&=\langle\mathrm{Hess} \, \md^K_S(\sigdis{p\cdot})(\gamma'(t)),\gamma'(t)\rangle\\
&\leq (1-K\md^K_S(\sigdis{p\gamma(t)}))\underbrace{\langle\gamma'(t),\gamma'(t)\rangle}_{=-1}\\
&=-1+Kf\,.
\end{align*}
In the case of lower timelike sectional curvature bounds, Hessian comparison goes in the other direction and the remaining calculations can be done in the same way.
\end{proof}

\begin{Remark}[Previous work]
\begin{enumerate}
    \item[]
    \item In \cite[Thm.\ 3.4]{harris1982triangle}, the author proves (in the language we use) that if $(M,g)$ is a globally hyperbolic (resp.\ order ($\sqrt{-K}$)-globally hyperbolic) spacetime with smooth timelike sectional curvature bounded above by $K \in \R$, then it satisfies a global synthetic timelike sectional curvature bound below by $K \in \R$ in the sense of hinge comparison. The restriction on sidelengths required there is included in the notion of size bounds: configurations so large that comparison configurations in $\mathbb{L}^2(K)$ do not exist are not considered. We obtain this result from Theorem \ref{Theorem: Smoothimpliessynthetic} and the globalization Theorem \ref{Theorem: Globalization}(ii).
    \item In \cite{alexander2008triangle}, the authors prove compatibility of smooth and synthetic sectional curvature bounds for general semi-Riemannian manifolds by studying a modified distance function and the corresponding Riccati comparison theory for its Hessian. However, the sectional curvature bounds they use also include bounds for spacelike planes, and their Jacobi field comparison methods use spacelike distances, so our purely Lorentzian results are not a consequence of their semi-Riemannian ones.
\end{enumerate}
\end{Remark}

\subsection{Timelike sectional curvature bounds and non-timelike planes}
\label{subsection: timelikeboundsandnontimelikeplanes}

We now show via an explicit example that smooth timelike sectional curvature bounds are strictly weaker than the semi-Riemannian sectional curvature bounds introduced in \cite{alexander2008triangle}. To this end, let $(M,g)$ be an FLRW spacetime satisfying the Einstein equations for the perfect fluid with energy density $\rho$ and pressure $p$. Then due to the isotropy of spatial slices, the sectional curvature of an arbitrary timelike $2$-plane $\Pi$ can be expressed as (see \cite[Sec.\ 3]{galloway2018existence})
\begin{align*}
    K(\Pi) = -\frac{8 \pi}{3} \left(C \frac{3(\rho + p)}{2} - \rho\right),
\end{align*}
where $C = C(\Pi) \geq 1$ is some constant (which can be arbitrarily large). Let now $\rho < 0$ and $p:=-\tilde{C}\rho$ (for some fixed constant $\tilde{C} \geq 1$), then $K(\Pi) \leq 0$. By arbitrariness of $\Pi$, $(M,g)$ thus has a smooth timelike sectional curvature bound from above by $0$. However, if $e_i$ ($i=1,2,3$) denotes an orthonormal basis of the spatial slice, then $\tilde{\Pi}:=\mathrm{span}(e_i,e_j)$ is a spacelike $2$-plane, and (see again \cite[Sec.\ 3]{galloway2018existence})
\begin{align*}
    K(\tilde{\Pi}) = \frac{8 \pi}{3} \rho < 0.
\end{align*}
We conclude that $(M,g)$ has smooth timelike sectional curvature bounded above by $0$, but it does not have sectional curvature bounded below by $0$ in the sense of \cite{alexander2008triangle}. For an explicit example of an FLRW spacetime with $\rho < 0$ and $p= - \rho$, let $f$ be a solution of the initial value problem $f'' f - f'^2 - 2 = 0$, $f(0) = 1$, $f'(0) = 1/4$. Let $(-\varepsilon,\varepsilon)$ be such that $f'(t) > 0$ and $|f'(t)| \leq 1/2$ for all $t \in (-\varepsilon,\varepsilon)$ and consider the FLRW spacetime $(-\varepsilon,\varepsilon) \times_f H^3$. It is then easy to check (see \cite[Thm.\ 12.11]{o1983semi}) that $\rho < 0$ and $p=-\rho$. Note that our example is not in contradiction to the equivalence of the semi-Riemannian bounds for some nonpositive constant (not necessarily $0$, as the example shows) in the sense of \cite{alexander2008triangle} and the strong energy condition on FLRW spacetimes \cite[Ex.\ 7.3]{alexander2008triangle}. The latter, in turn, is known to be equivalent to a smooth timelike sectional curvature bound from above by $0$ \cite[Sec.\ 3]{galloway2018existence}. We conclude that, on FLRW spacetimes, the strong energy condition is equivalent to a smooth timelike sectional curvature bound from above by $0$ and a lower bound by some nonpositive (not necessarily $0$) constant on the sectional curvatures of spacelike tangent $2$-planes.


\section{Conclusion \& outlook}\label{section: conclusion_outlook}

In this work, we have established the equivalence of smooth timelike sectional curvature bounds from above (resp.\ below) with (local) synthetic timelike curvature bounds below (above), see Theorems \ref{Theorem: syntheticimpliessmooth} and \ref{Theorem: Smoothimpliessynthetic}. We recalled that local synthetic bounds globalize under appropriate conditions, see Theorem \ref{Theorem: Globalization}, thus obtaining the hinge comparison result from \cite{harris1982triangle} as a consequence. As an essential tool, we utilized the Hessian comparison estimate for the time separation and signed distance functions, see Theorem \ref{Theorem: Hessiancomparison}. Finally, in Subsection \ref{subsection: timelikeboundsandnontimelikeplanes}, we saw that timelike sectional curvature bounds do not imply the corresponding semi-Riemannian sectional curvature bounds from \cite{alexander2008triangle}.

One immediate consequence of Theorem \ref{Theorem: Smoothimpliessynthetic} is a relaxation in the assumptions of the Lorentzian Reshetnyak gluing theorem \cite[Thm.\ 5.2.1]{beran02}, which states that when gluing spacetimes with upper curvature bounds in the sense of \cite{alexander2008triangle} one obtains a Lorentzian pre-length space with upper curvature bounds in the sense of triangle comparison. In fact, due to results from \cite{rott2022gluing}, this can be improved even further, one even obtains a strongly causal Lorentzian length space. We refer to \cite{beran02} for notions below not introduced in this article.

\begin{Theorem}[Improved version of the Lorentzian Reshetnyak gluing theorem]\label{thm:improvedReshetnyak}
Let $(M_1,g_1)$ and $(M_2,g_2)$ be strongly causal spacetimes with $\dim(M_1)=:n \geq m := \dim(M_2) \geq 2$. 
Let $A_1$ and $A_2$ be two closed non-timelike locally isolating subsets of $M_1$ and $M_2$, respectively. 
Let $f:A_1 \to A_2$ be a $\tau$-preserving and $\leq$-preserving locally bi-Lipschitz homeomorphism which locally preserves the signed distance. Suppose $A_1$ and $A_2$ are convex in the sense of \cite[Remark 5.1.1(iii)]{beran02}. If $M_1$ and $M_2$ have smooth timelike sectional curvature bounded \emph{below}\footnote{Recall that due to the convention of reversing the curvature inequality for timelike planes, it is now justified to talk about lower timelike curvature bounds, as this is the correct inequality corresponding to upper sectional curvature bounds in the sense of \cite{alexander2008triangle}.} by $K \in \R$, then the Lorentzian amalgamation $M_1 \sqcup_A M_2$ is a \emph{strongly causal Lorentzian length space} with timelike curvature bounded above by $K$. 
\end{Theorem}
\begin{proof}
Due to Theorem \ref{Theorem: Smoothimpliessynthetic} and \cite[Prop.\ 3.5, Ex.\ 3.24, Thm.\ 3.26 \& Ex.\ 4.9]{kunzinger2018lorentzian}, $M_1$ and $M_2$ are strongly causal regular (SR)-localizable Lorentzian length spaces with timelike curvature bounded above by $K$. We want to apply \cite[Thm.\ 5.2.1]{beran02}, then \cite[Thm.\ 4.9 \& Thm.\ 5.11(iv)]{rott2022gluing} yield the claim. The proof of \cite[Thm.\ 5.2.1]{beran02} only uses non-timelike curvature bounds in its reference to \cite[Lem.\ 4.3.3]{beran02}, whose proof is analogous to that of \cite[Lem.\ 4.3.1]{beran02}, except that it uses one sided triangle comparison in triangles with one spacelike side, i.e.\ a spacelike geodesic $\gamma$ from $p_2$ to $p_3$ in the timelike future (or past) of a point $p_1$, and comparing $\tau(p_1,\gamma(t))$ to the corresponding length in a comparison triangle for $p_1p_2p_3$.

So we only need to prove this triangle comparison. This works via Theorem \ref{Theorem: Hessiancomparison}(i): consider $f(t)=\md^K_S(\sigdis{p\gamma(t)})=-\md^K_\tau(\tau(p,\gamma(t)))$, then 
\begin{equation*}
f''(t)=\langle\mathrm{Hess} \, \md^K_S(\sigdis{p\cdot}) (\gamma'(t)),\gamma'(t)\rangle \leq (1-Kf)\underbrace{\langle\gamma'(t),\gamma'(t)\rangle}_{=1}
\end{equation*}

This corresponds to \cite[Eq.\ (5.6)]{Kir18} (note that $h=f$ here), and from there on, we can follow the rest of the proof of \cite[Thm.\ 5.2.4]{Kir18}.
\end{proof}

Even in the setting of smooth spacetimes with timelike sectional curvature bounds, there still remain interesting open questions, many of which would be even more impactful under timelike Ricci bounds due to the connection to General Relativity, but are more reasonably achievable under the stronger sectional curvature condition. An example would be the analogue of Cheeger--Colding's quantitative almost splitting result \cite{cheeger1996lower}, a proof of which would also indicate the ``correct" notion of Lorentzian Gromov--Hausdorff convergence, which is a very present topic of conversation in the research of Lorentzian synthetic spaces. Efforts in this direction are currently under way. Another interesting question is whether one can find the \textit{soul} of a spacetime under timelike sectional curvature bounds, similar to the famous Cheeger--Gromoll result \cite{cheeger1972structure}.

\section*{Acknowledgments}
This research was funded in part by the Austrian Science Fund (FWF) [Grant DOI \linebreak \href{https://doi.org/10.55776/PAT1996423}{10.55776/PAT1996423}
and \href{https://doi.org/10.55776/P33594}{10.55776/P33594}]. For open access purposes, the authors have applied a CC BY public copyright license to any author accepted manuscript version arising from this submission. Argam Ohanyan is also supported by the ÖAW-DOC scholarship of the Austrian Academy of Sciences. The authors would like to thank the anonymous referee for their careful reading of the article and their suggestions.

\addcontentsline{toc}{section}{References}

\end{document}